\documentclass[x11names,oneside,11pt]{amsart}

\usepackage[T2A,T1]{fontenc}
\usepackage[utf8]{inputenc}
\usepackage[english]{babel}

\usepackage{xcolor}
\usepackage{hyperref}

\definecolor{myblue}{RGB}{60, 80, 197}
\hypersetup{%
  colorlinks=true,
  linkcolor=OrangeRed2,
  citecolor = SlateBlue2,
  urlcolor = MediumPurple4
}

\usepackage{tabularx}
\usepackage{array}      
\usepackage{ragged2e}

\usepackage{mathtext}
\usepackage{amsmath,amsthm,mathtools,amssymb} 

\theoremstyle{definition}

\newtheorem{theorem}{Theorem}

\newtheorem{corollary}{Corollary}

\newtheorem{question}{Question}

\newtheorem{lemma}{Lemma}
\theoremstyle{remark}
\newtheorem{remark}{Remark}

\makeatletter
\newtheorem*{rep@theorem}{\rep@title}
\newcommand{\newreptheorem}[2]{%
\newenvironment{rep#1}[1]{%
 \def\rep@title{{\bf #2 \ref{##1}}}%
 \begin{rep@theorem}}%
 {\end{rep@theorem}}}
\makeatother
\newreptheorem{theorem}{Theorem}

\usepackage[left=30mm, right=30mm, bottom=25mm, top=25mm,includefoot]{geometry}

\DeclareMathOperator{\Sep}{Sep}

\DeclareMathOperator{\conj}{conj}

\DeclareMathOperator{\Sym}{Sym}
\DeclareMathOperator{\Pic}{Pic}
\DeclareMathOperator{\Ker}{Ker}
\DeclareMathOperator{\supp}{supp}

\newcommand{\eqdef}{\vcentcolon=}

\newcommand{\N}{\mathbb{N}}
\newcommand{\Z}{\mathbb{Z}}

\newcommand{\R}{\mathbb{R}}

\newcommand{\PP}{\mathbb{P}}
\newcommand{\rat}{\mathbf{r}}

\usepackage{epigraph}
\usepackage{textcase}
\usepackage{version}
\excludeversion{confidential}

\usepackage{caption}
\usepackage[colorinlistoftodos]{todonotes}

\usepackage{tikz}
\usepackage{tikz-cd} 
\usetikzlibrary{cd}
\usepackage{pgfplots}
\pgfplotsset{width=6cm,compat=newest}

\usepackage{pdfpages}

\newcommand{\rom}[1]{\uppercase\expandafter{\romannumeral #1\relax}} 

\makeatletter
\renewcommand{\section}{%
  \@startsection{section}{1}%
    \z@{.7\linespacing\@plus\linespacing}{.5\linespacing}%
    {\normalfont\bfseries\centering}%
}
\renewcommand{\@seccntformat}[1]{%
  \ifcsname the#1\endcsname
    \bfseries\S\csname the#1\endcsname.\ 
  \fi
}
\makeatother

\begin{document}   
   \title[On finiteness properties of the separating semigroup of a real curve]{On finiteness properties of the separating semigroup of a real curve}

  \author{Matthew Magin}

  \date{\today}

  \address{Matthew Magin: Saint Petersburg University, 7/9 Universitetskaya nab., St. Petersburg, 199034 Russia}

  \email{matheusz.magin@gmail.com}

  \thanks{This work was performed at the Saint Petersburg Leonhard Euler International Mathematical Institute and supported by the Ministry of Science and Higher Education of the Russian Federation (agreement no. 075–15–2025–343).}

  \begin{abstract}
  A real morphism $f$ from a real algebraic curve $X$ to $\PP^1$ is called separating if $f^{-1}(\R \PP^1) = \R X$. A separating morphism defines a covering $\R X \to \R \PP^1$. Let $X_1, \ldots, X_r$ denote the components of $\R X$. M. Kummer and K. Shaw~\cite{kummer_separating_2020} defined the separating semigroup of a curve $X$ as the set of all vectors $d(f) = (d_1(f), \ldots, d_r(f)) \in \N^{r}$ where $f$ is a separating morphism $X \to \PP^1$ and $d_i(f)$ is the degree of the restriction of $f$ to $X_i$. 
  
  In the present paper we prove that for a non-negative integer number $g$ the set of all separating semigroups of genus $g$ curves is finite. 
  \end{abstract}

  \maketitle

  \section{Introduction}
  
  By a \emph{real curve} we mean a complex algebraic curve $X$ equipped with an antiholomorphic involution $\conj\colon X \to X$. Its real point set is $\R X \eqdef \{ p \in X \ \vert \ \conj(p)=p\}$. All curves considered here are smooth, projective and irreducible. 
  
  \medskip 
  
  A real curve $X$ is called \emph{separating} if the space $X \setminus \R X$ is disconnected. In this case $X \setminus \R X$ consists of two connected components which are interchanged by an anti-holomorphic involution. The boundary orientation induced on $\R X$ by one of the halves of $X \setminus \R X$ is called {\em a complex orientation}. 
  
  Ahlfors~\cite{ahlfors_open_1950} proved that a real curve $X$ is separating if and only if there exists a \emph{separating morphism} $f\colon X \to \PP^1$, that is a morphism such that $f^{-1}(\R \PP^1)= \R X$. Such a morphism defines a covering map $f\vert_{\R X}\colon \R X \to \R \PP^1$. Let $X_1, \ldots, X_{r}$ denote the connected components of $\R X$. Denoting by $d_i(f)$ the degree of the restriction of $f$ to $X_i$ we may associate every separating morphism $f$ with a vector $d(f) \eqdef (d_1(f), \ldots, d_r(f))$. M. Kummer and K. Shaw~\cite{kummer_separating_2020} defined the \emph{separating semigroup} of the curve $X$ as the set of all such vectors:
  \[
    \Sep(X) \eqdef \{ d(f) \ \vert \ f\colon X \to \PP^1 \text{~--- separating morphism}\}. 
  \] 
  
  It is easy to check that $\Sep(X)$ is indeed an additive semigroup, see~\cite[Prop. 2.1]{kummer_separating_2020}. 
  
  Without going into details, let us note that   the separating semigroup has now been described for the following cases: \emph{M-curves}, i.e. real curves of genus $g$ with maximal number of components $b_0(\R X) = g+1$, (see~\cite[Theorem 1.7]{kummer_separating_2020}), curves of genus $\leqslant 4$ (see~\cite[Theorem 1.7]{kummer_separating_2020},~\cite{orevkov_separating_2019},~\cite[Example 3.3]{orevkov_algebraically_2021},~\cite{orevkov_separating_2025}), hyperelliptic curves (see~\cite{orevkov_separating_2019} and~\cite[\S 4]{orevkov_separating_2025}) and plane quintics (see~\cite{magin_orevkov_2025}). 

  \medskip


  Before stating the main results of the paper, we introduce the following definition. A totally real divisor $P = p_1 + \ldots + p_n$ on a real curve $X$ is called {\em separating} if there exists a separating morphism $f\colon X \to \PP^1$ and a point $p_0 \in \R \PP^1$ such that $P = f^{-1}(p_0)$. In this case, we say that $P$ {\em has degree partition} $d(P) \eqdef d(f) \in \Sep(X)$, i.e. $d_i(P) \eqdef \deg(P\vert_{X_i})$. It will be more convenient in what follows to work with separating divisors rather than separating morphisms.

  \subsection*{Notation} We denote:
  \[
      \N = \{ n \in \Z \ \vert \ n \geqslant 1 \}, \quad \N_0 = \{ n \in \Z \ \vert \ n \geqslant 0 \}.
  \] 
  Below $g$ always denotes the genus of a real curve $X$ and $r$ denotes the number of connected components of $\R X$. We always equip the sets $\N^r$ and $\N_0^{r}$ with the standard structure of an additive semigroup. For $d \in \N^{r}$ denote $|d|\eqdef \sum_{i=1}^{r}d_i$. 
  
  \subsection*{Main results and outline} The main goal of this article is to prove the following theorem.

  \begin{reptheorem}{main_thm}
    For a non-negative integer number $g$ the set of all separating semigroups of genus $g$ curves is finite. 
  \end{reptheorem}

  The proof of this theorem relies on two results, each of which we believe to be of independent interest.

  First, in \S\ref{Deleting_section} we show that from any separating divisor $P$ with $h^0(P) \geqslant 3$ (in particular, whenever $\deg P \geqslant g+2$) one can remove a non-empty subset of its points so that the resulting divisor remains separating and contains at least half of the original points.

  \begin{reptheorem}{deleting_points}
    Let $P = p_1 + \ldots + p_n$ be a separating divisor on a real curve $X$ and suppose that $h^0(P) \geqslant 3$. Then there exists a proper subset  $I \subsetneq \{ 1, \ldots, n\}$ such that the divisor $Q \eqdef \sum_{i \in I} p_i$ is separating and satisfies $\deg{Q} \geqslant \lceil \frac{n}{2} \rceil$. 
  \end{reptheorem}
  
  The Abel-Jacobi theorem plays a key role in the proof of Theorem~\ref{deleting_points}. 

  Also note that, combined with Ahlfors' existence theorem, Theorem~\ref{deleting_points} recovers his classical bound on the degree of a separating morphism: for a separating curve of genus $g$, the minimal degree of a separating morphism $X \to \PP^1$ is at most $g+1$. Moreover, Theorem~\ref{deleting_points} implies the following refinement: for a given separating divisor $P$ one can find a separating divisor $D \leqslant P$ with $\deg{D} \leqslant g+1 - h^1(P)$.

  \medskip 

  Recall that a divisor $D$ on an algebraic curve $C$ is called \emph{non-special} if $h^1(D) = 0$ and \emph{special} otherwise.  By Serre duality, this is equivalent to $h^0(K_C - D) = 0$, where $K_C$ denotes the canonical divisor of $C$.  Note that if $D$ is non-special, then $D + p$ is non-special for any point $p \in C$. Indeed, $h^0(K_C - D - p) \leqslant h^0(K_C-D)=0$, so $h^0(K_C -(D+p))=0$. 

  \medskip 

  In \S\ref{sep_structure} we prove that the separating semigroup of any real curve $X$ with $b_0(\R X) = r$ can be expressed as
     \begin{equation}\tag{$\diamondsuit$}
          \Sep(X) = \Sep_s(X) \cup \bigcup_{i=1}^{m} (d(P_i) + \N_0^{r}).  \label{sep_decomp} 
      \end{equation}
  where $\Sep_s(X)$ is a finite set of degree partitions of special separating divisors, and $d(P_i)$ are minimal (with respect to the pointwise partial order on $\Sep(X) \subset \N^r$) degree partitions corresponding to non-special separating divisors $P_1, \ldots, P_m$.   

  The decomposition~\eqref{sep_decomp} shows that the separating semigroup of any real curve  $X$ is completely determined by two finite sets: the finite set $\Sep_s(X)$ and the minimal non-special separating divisors $P_1,\ldots, P_m$. Using this description, we observe that the standard notion of finite generation is somewhat meaningless in this setting. Indeed, we prove that if $b_0(\R X) \geqslant 2$, then $\Sep(X)$ is {\bf not} finitely generated as a semigroup. We also provide simple examples illustrating how known descriptions of separating semigroups can be recast in this new framework.

  Finally, in \S\ref{main_theorem_proof} we prove the main Theorem~\ref{main_thm}. The key step of the proof is the uniform bound on degrees of minimal non-special separating divisors $P_1, \ldots, P_m$, which we derive using Theorem~\ref{deleting_points}.

  \section{Deleting points from separating divisors}\label{Deleting_section}

  In this section, $X$ denotes a separating real curve of genus $g$, and $\omega_1, \ldots, \omega_g$ denote a basis of the space of holomorphic 1-forms on $X$. For each $n\geqslant1$, let $\varphi_n\colon\Sym^n(X)\to\Pic^n(X)$ denote the Abel--Jacobi map.
  
  The main ingredient of the proof of Theorem~\ref{deleting_points} is the following lemma, which is a combination of the infinitesimal version of the Abel-Jacobi theorem and Lemma 2.10 from~\cite{kummer_separating_2020}.

   \begin{lemma}[Lemma 3.2 in~\cite{orevkov_separating_2019}]\label{abel-jacobi_lemma}

   Let $p_1, \ldots, p_n \in \R X$ be pairwise distinct points. Then the divisor $P = p_1 + \ldots + p_n$ is separating if and only if there exists a tuple of positive (with respect to some complex orientation) tangent vectors $v_i \in T_{p_i}(\R X), \ i = 1, \ldots, n $  such that 
   \begin{equation}\tag{$\dagger$} 
       \sum_{i=1}^{n}\omega_k(v_i) = 0 \quad \text{ for all } k = 1, \ldots, g.  \label{abel_jacobi_lemma_eq}
    \end{equation}
  \end{lemma}

  \begin{proof}
    The tuple $v = (v_1, \ldots, v_n)$  defines a tangent vector to $\Sym^n(X)$ at $P$.  Condition~\eqref{abel_jacobi_lemma_eq}  is precisely the statement that $v \in \Ker(d_{P}\varphi_n)$. Thus, the implication $(\Rightarrow)$ follows directly from the Abel-Jacobi theorem. 

    Let us prove the implication $(\Leftarrow)$. First note that $\Ker(d_P\varphi_n)=T_P|P|$. Indeed, we have $T_P|P| \subseteq \Ker(d_{P}\varphi_n)$ and both spaces have dimension $\dim|P|=h^0(P) -1$, see e.g.~\cite[p. 245]{griffiths_principles_2014}. Thus, $v$ is tangent to $|P|$ at $P$ and hence there exists a smooth real deformation
    \[
      p_i(t)\colon [0,1] \to \R X\colon \quad P(t) = \sum_{i=1}^n p_i(t) \in |P|, \quad  p_i(0) = p_i, \quad \frac{d}{dt}\bigg\vert_{t=0} \ p_i(t) = v_i.
    \]
    
    Since all $v_i$ are positive with respect to the chosen complex orientation on $\R X$, for sufficiently small $t \in [0,1]$ all points $p_i(t)$ move in the direction of the complex orientation. Hence, for sufficiently small $t_0>0$, the divisors $Q = P(t_0)$ and $P$ interlace (that is, every component of $\R X \setminus Q$ contains exactly one point of $P$ and vice versa). Since $Q \in |P|$, there exists a real meromorphic function $f$ such that $(f) = P - Q$. Since $Q$ and $P$ interlace, by~\cite[Lemma 2.10]{kummer_separating_2020} the meromorphic function $f$ is a separating morphism.
  \end{proof} 

  \begin{remark}
    In~\cite{orevkov_separating_2019} the ($\Leftarrow$) part of Lemma 3.2 assumes that the divisor is non-special. Orevkov uses the fact that for a non-special divisor $P$,  the Abel--Jacobi map $\varphi_n$ is a submersion near $P$ (this need not hold for a special divisor). From that he obtains that the vector $v$ is tangent to $\varphi_n^{-1}(\varphi_n(P)) = |P|$ at $P$. But, as we see in the proof of Lemma~\ref{abel-jacobi_lemma}, the equality $T_{P}|P| = \Ker(d_{P}\varphi_n)$ (which holds for any effective divisor) implies that $v$ is tangent to $|P|$ at $P$. Thus, the assumption that $P$ is non-special is, in fact, redundant.  
  \end{remark}

  \begin{theorem}\label{deleting_points}
      Let $P = p_1 + \ldots + p_n$ be a separating divisor on a real curve $X$ and suppose that $h^0(P) \geqslant 3$. Then there exists a proper subset  $I \subsetneq \{ 1, \ldots, n\}$ such that the divisor $Q \eqdef \sum_{i \in I} p_i$ is separating and satisfies $\deg{Q} \geqslant \lceil \frac{n}{2} \rceil$. 
  \end{theorem}  

  \begin{proof}
    Fix a complex orientation on $\R X$. Since $P$ is separating, there exist tangent vectors $v_i \in T_{p_i}(\R X), \ i = 1, \ldots, n$ satisfying the conditions of Lemma~\ref{abel-jacobi_lemma}. In particular, $v = (v_1,\ldots,v_n) \in \Ker(d_{P}\varphi_n)$.  Since $\dim{\Ker(d_{P}\varphi_n)} = h^0(P)-1$  and $h^0(P) \geqslant 3$, we have
    \[
      \dim{\Ker(d_{P}\varphi_n)} \geqslant 2.
    \] 

    Therefore we can choose a real tangent vector $w = (w_1, \ldots, w_n) \in \Ker(d_{P}\varphi_n)$ not proportional to $v$. Since each space $T_{p_i}(\R X)$ is one-dimensional, we can write
    \[
        w_i = \lambda_i \cdot v_i, \text{ where } \lambda_i \in \R.
    \]

    By our choice of $w$, the numbers $\lambda_i$ are not all equal. Replacing $w$ by $-w$ if necessary, we may therefore assume that their minimum is attained at most $\lfloor n/2\rfloor$ times.

    Set
    \[
        \mu=\min_i\lambda_i,
        \qquad
        u=w-\mu v,
        \qquad
        I=\{i \ \vert \ \lambda_i>\mu\}.
    \]

    Clearly, $u$ lies in $\Ker(d_{P}\varphi_n)$. Since $u_i = (\lambda_i - \mu) v_i$, the vector $u_i$ is positive with respect to the chosen complex orientation for $i\in I$ and zero otherwise. Moreover,
    \[
        \left\lceil\frac n2\right\rceil \leqslant \#I <n.
    \]

    Put $k=\# I$ and write $I = \{ i_1, \ldots, i_k\}$.  Consider the divisor $Q=\sum_{i\in I}p_i$ and the tuple of tangent vectors $u'= (u_{i_1}, \ldots, u_{i_{k}})\in T_Q\Sym^k(X)$. We claim that $u'$ lies in $\Ker(d_{Q}\varphi_k)$. Indeed, since $u \in \Ker(d_{P}\varphi_n)$, for every holomorphic 1-form $\omega$ on $X$ we have 
    \[
        \sum_{i=1}^n \omega(u_i) = 0.
    \]
    Since $u_i = 0$ for $i \notin I$, we have
    \[
        0=\sum_{i=1}^n \omega(u_i)  = \sum_{i \in I} \omega(u_i),
    \]
    which implies that $u' \in \Ker(d_{Q}\varphi_k)$.  Hence Lemma~\ref{abel-jacobi_lemma} implies  that $Q$ is separating. Finally,
    \[
        \deg Q=\# I \geqslant\left\lceil\frac n2\right\rceil,
    \]
    as required.
  \end{proof} 

  \begin{remark}\label{degree_g+2}
    Note that by the Riemann--Roch theorem, the hypothesis $h^0(P)\geqslant3$ of Theorem~\ref{deleting_points} holds whenever $\deg P\geqslant g+2$.
  \end{remark}

  Recall that the group of divisors on $X$ admits a natural
  partial order defined by
  \[
      D_1\geqslant D_2
      \quad\Longleftrightarrow\quad
      D_1-D_2 \text{ is effective}.
  \]
  We call a separating divisor \emph{minimal} if it is minimal with respect to this order among separating divisors.

   \begin{corollary}
    Let $D = p_1 + \ldots + p_n$ be a minimal separating divisor on $X$. Then $h^0(D) = 2$. If, in addition, $D$ is special, then we have $\deg{D} \leqslant g$.
  \end{corollary}
  \begin{proof}
    The first statement follows directly from Theorem~\ref{deleting_points}. The second statement follows from the Riemann--Roch theorem. 
  \end{proof}

  Theorem~\ref{deleting_points} also yields the following degree bound.

  \begin{corollary}\label{ahlfors_refinement}
  Let $P$ be a separating divisor on a real curve $X$ of genus $g$.
  Then there exists a separating divisor $D\leqslant P$ such that
  \[
      h^0(D)=2,
      \qquad
      \deg D\leqslant g+1-h^1(P).
  \]
  In particular, if $P$ is special, then $D$ can be chosen
  with $\deg D\leqslant g$.
  \end{corollary}

  \begin{proof}
  Choose a minimal separating divisor $D\leqslant P$. Theorem~\ref{deleting_points} implies that $h^0(D)=2$. Since $D\leqslant P$, we have $h^1(D)\geqslant h^1(P)$.
  The Riemann--Roch theorem therefore gives
  \[
      \deg D=g+1-h^1(D)\leqslant g+1-h^1(P).
  \]
  \end{proof}

  In particular, every minimal separating divisor has degree
  at most $g+1$. Combined with Ahlfors' existence theorem,
  this recovers his classical degree bound.

  \begin{corollary}[see \S4 in~\cite{ahlfors_open_1950}]
  \label{ahlfors_g+1}
  Let $X$ be a separating real curve of genus $g$.
  Then the minimal degree of a separating morphism
  $f\colon X\to\PP^1$ is at most $g+1$.
  \end{corollary}

  \begin{proof}
  Choose any separating divisor $P$ on $X$ (it exists by the result of Ahlfors~\cite[\S 4.2, pp.~124--126]{ahlfors_open_1950}). Corollary~\ref{ahlfors_refinement} gives a separating divisor
  $D\leqslant P$ with
  \[
      \deg D\leqslant g+1-h^1(P)\leqslant g+1.
  \]
  The separating morphism having $D$ as a fiber has the
  required degree.
  \end{proof}

  Note that the bound in Corollary~\ref{ahlfors_g+1} is sharp
  only for $M$-curves.
  Gabard~\cite{gabard_sur_2006} improved this bound:
  he showed that for a separating real curve with
  $b_0(\R X)=r$, the value $g+1$ can be replaced by
  $\frac{g+r+1}{2}$.
  Later, Coppens~\cite{coppens_separating_2013} showed that
  Gabard's bound is sharp.

  Unlike Gabard's bound, Corollary~\ref{ahlfors_refinement}
  concerns a prescribed separating divisor $P$ and guarantees
  a separating divisor $D\leqslant P$, with a degree bound
  depending on $h^1(P)$.

  \section{Structure of separating semigroups}\label{sep_structure}

  \subsection{Characterization of the separating semigroup via two finite sets}

  \begin{lemma}[see Prop. 3.2 in~\cite{kummer_separating_2020}]\label{KS_non_special}
    Let $P$ be a separating divisor on a real curve $X$. Suppose that for some point $p \in \R X \setminus P$ we have $h^0(P+p) > h^0(P)$. Then the divisor $P + p$ is also separating. 

    In particular, if $P$ is non-special, then $d(P) + \N_0^{r} \subset \Sep(X)$. \qed
  \end{lemma}

  The following elementary lemma is proven, for instance, in~\cite[Prop. 2.1]{huisman_non-special_2003}. For the reader's convenience, we include here a proof which we find illustrative.   
    
  \begin{lemma}\label{canonical_even_degree}
  Let $X$ be a real curve and $P$ a real canonical divisor on $X$. Then $P$ has even degree on every connected component of $\R X$.
  \end{lemma}
    \begin{proof}
       Since $P$ is canonical, it is the divisor of a meromorphic 1-form $\omega$ on $X$. We may choose $\omega$ to be real. The form $\omega$ does not vanish on $\R X \setminus \supp(P)$, and thus defines an orientation on this set. Let $X_i$ be a connected component of $\R X$.  The complement $X_i \setminus \supp(P)$ is a disjoint union of open arcs; two adjacent arcs are oriented in the same direction if and only if the multiplicity of their common boundary point in $P$ is even.  Since $X_i$ is a topological circle, the total number of orientation reversals along $X_i$ must be even. Consequently, the sum of the multiplicities of $P$ on $X_i$, i.e., $\deg(P|_{X_i})$, is even.
  \end{proof}

  To prove the next lemma, we introduce the pointwise partial order $\preceq$ on $\N^r$.  That is, for $u, v \in \N^r$, we say that $u \preceq v$ if $u_i \leqslant v_i$ for all $i = 1, \ldots, r$.

  \begin{lemma}\label{main_lemma}
  The separating semigroup of any real curve $X$ of genus $g$ with $b_0(\R X) = r$ admits the decomposition
  \[
    \Sep(X) = \Sep_s(X) \,\cup\, \bigcup_{i=1}^{m} \bigl( d(P_i) + \N_0^{\,r} \bigr),
  \]
  where $\Sep_s(X)$ is the finite set of degree partitions of special separating divisors, and $d(P_i)$ are the minimal (with respect to the pointwise partial order $\preceq$) degree partitions  corresponding to non-special separating divisors $P_i$.

  Moreover, we have
  \[
    |\Sep_s(X)| \leqslant \binom{2g - 3}{r} + \binom{g - 2}{r - 1}.
  \]
  \end{lemma}
   \begin{proof}
   Indeed, the semigroup $\Sep(X)$ decomposes as
  \[
    \Sep(X) = \{ d(P) \mid P \text{ is separating} \} = \Sep_s(X) \cup \Sep_n(X),
  \]
  where
  \begin{itemize}
    \item $\Sep_s(X) \coloneqq \{ d(P) \mid P \text{ is separating and special} \},$
    \item $\Sep_n(X) \coloneqq \{ d(P) \mid P \text{ is separating and non-special} \}.$
  \end{itemize}

    Since $\deg K_X = 2g - 2$, it follows from Serre duality that any effective divisor of degree greater than $2g - 2$ is non-special. Hence, for every $d \in \Sep_s(X)$ we have $|d| \leqslant 2g - 2$.  Moreover, any special divisor of degree $2g - 2$ is canonical. By Lemma~\ref{canonical_even_degree}, a canonical divisor has even degree on each connected component of $\R X$. Consequently, we obtain the bound
    \begin{equation}\tag{$\diamond$}
      \#\Sep_s(X) \leqslant \binom{2g - 3}{r} + \binom{g - 2}{r - 1}.\label{ineq_sep}
    \end{equation}

    The first term on the right-hand side of~\eqref{ineq_sep} counts the number of tuples $(d_1, \ldots, d_r) \in \N^r$ with $\sum_{i=1}^r d_i \leqslant 2g - 3$, while the second term counts those with $\sum_{i=1}^r d_i = 2g - 2$ and all $d_i$ even.
        
    Now consider the set $\Sep_n(X)$.  Let $d_1, \dots, d_m$ be the minimal elements of $\Sep_n(X)$ with respect to the pointwise partial order $\preceq$. There are finitely many of them due to  Dickson's lemma~\cite{Dickson_1913}.\footnote{We will later give an estimate for $m$.} Let $P_1, \ldots, P_m$ be non-special separating divisors such that $d(P_i) = d_i$ for $i = 1, \ldots, m$.  Since each $P_i$ is non-special, Lemma~\ref{KS_non_special} implies that $d(P_i) + \N_0^{\,r} \subset \Sep(X)$ for all $i = 1, \ldots, m$.  Moreover, every element of $\Sep_n(X)$ dominates some $d(P_i)$, and hence
    \[
      \Sep_n(X) = \bigcup_{i=1}^{m} \bigl( d(P_i) + \N_0^{\,r} \bigr).
    \]
  \end{proof}   

  The proof of Lemma~\ref{main_lemma} raises the following two questions.

  \begin{question}
    Is it possible that, for some real curve $X$, a degree partition $d \in \Sep(X)$ is realized by two separating divisors $P$ and $Q$ such that $P$ is special and $Q$ is non-special?
  \end{question}

  \begin{question}
  Suppose that a canonical divisor $P$ on $X$ is separating. Which degree partitions $d(P) = (d_1, \ldots, d_r)$ can occur? We know that $|d(P)| = 2g - 2$ and that each $d_i$ is even by Lemma~\ref{canonical_even_degree}. Are there any other restrictions?
  \end{question}  

  \subsection{Infinite generation}

  Note that Lemma~\ref{main_lemma} demonstrates that the usual notion of finite generation is of limited relevance in this context.

  \begin{corollary}\label{infin_gen}
  For any separating real curve $X$ with $r = b_0(\R X) \geqslant 2$, the semigroup $\Sep(X)$ is not finitely generated.
  \end{corollary}

  \begin{remark}
  The case $r = 1$ is uninteresting, since every additive subsemigroup of $\N$ is finitely generated.
  \end{remark}

  \begin{proof}[Proof of corollary~\ref{infin_gen}]
    First, note that the case of arbitrary $r > 2$ reduces to the case $r = 2$ via the projection
  \[
    \pi \colon \N^r \to \N^2, \quad \pi(x_1, \ldots, x_r) = (x_1, x_2).
  \]
  Therefore, we assume that $r = 2$. By Lemma~\ref{main_lemma}, we have
  \[
    \Sep(X) = \Sep_s(X) \,\cup\, \bigcup_{i=1}^{m} \left( d^{(i)} + \N_0^{2} \right).
  \]

  Choose a vector $d^{(j)}$ with minimal first coordinate. Without loss of generality, we may assume that it is $d^{(1)}$.  Consider $d = d^{(1)} + (0,k)$, where $k \in \N$ will be chosen later.  Evidently, $d$ cannot be written as a sum of elements of $\Sep(X)$ whose terms include vectors from $\Sep_n(X)$.

  For $a = (a_1,a_2) \in \N^2$ denote $\rat(a) \eqdef a_2/a_1$. Note that the following elementary lemma holds.
    
    \begin{lemma}\label{lemma_1} 
      For all $a, b \in \N^2$, one has $\rat(a+b) \leqslant \max(\rat(a), \rat(b))$. 
    \end{lemma}
    
  Now suppose, for a contradiction, that $d$ belongs to the subsemigroup generated by $\Sep_s(X)$. 
    Write $\Sep_s(X) = \{ e^{(1)}, \ldots, e^{(m)} \}$ and $d = \sum_{j=1}^{m} \alpha_j \cdot e^{(j)}$ for some $\alpha_j \in \N_0$. By Lemma~\ref{lemma_1} we have 
    \begin{equation}\tag{$\star$}
        \rat(d) \leqslant \max\limits_{j \colon \alpha_j > 0}{\rat(\alpha_j e^{(j)})} = \max\limits_{j \colon \alpha_j > 0}{\rat(e^{(j)})} \leqslant \max_{j}\rat(e^{(j)})=C \label{ineq_1}
    \end{equation}
    
  On the other hand, by choosing $k$ large enough, we can make $\rat(d)$ arbitrarily large, which contradicts~\eqref{ineq_1}.  
  \end{proof}

  \subsection{Examples}\label{examples}

    As noted in the introduction, Lemma~\ref{main_lemma} shows that to describe the separating semigroup of a real curve $X$, it suffices to describe two finite sets: $\Sep_s(X)$ and the divisors $P_1, \ldots, P_m$ (in the notation of Lemma~\ref{main_lemma}).

    To illustrate this decomposition, we present several explicit examples in simple cases.
    
    \begin{enumerate}
      \item If $X$ is an $M$-curve, $\Sep_s(X) = \varnothing$ (see~\cite{huisman_geometry_2001}) and 
      \[
        \Sep(X) = (1, \ldots, 1) + \N_0^{g+1}.
      \]

      \item Let $X$ be a separating non-$M$-curve of genus 2. Then, as shown in~\cite[Example 2.5]{kummer_separating_2020}
      \[
        \Sep(X) = 2 + \N_0 = \{ 2 \}  \cup \left( 3 + \N_0 \right),
      \]
      where $\Sep_s(X) = \{ 2 \}$ corresponds to the hyperelliptic projection, which in this case coincides with the canonical map.

      \item Let $X$ be a separating non-hyperelliptic non-$M$-curve of genus 3. Then $X$ is a {\em hyperbolic quartic}. That is a plane real curve of degree 4 with two nested ovals. We label the ovals from inner to outer. As shown in~\cite{orevkov_separating_2019}\footnote{A simplified proof using completely different argument can be found in~\cite[Example 3.3]{orevkov_algebraically_2021}.}, $\Sep(X) = (1, 2)+\N_0^2$. It is easy to see that in this case $\Sep_s(X) = \{ (1, 2), (2, 2) \}$ and, consequently, 
      \[
        \Sep(X) = \{ (1, 2), (2, 2)\} \cup \left( (1, 3)+ \N_0^2 \right) \cup \left( (3, 2) + \N_0^2 \right). 
      \]

      Indeed, since every divisor of degree greater than $4 = 2\cdot 3 - 2$ on a curve of genus~3 is non-special, it remains only to show that the degree partition $(1,3)$ cannot arise from a special divisor. 

      This follows from the fact that the canonical linear system on a plane quartic is cut out by lines, while a divisor of degree partition $(1,3)$ cannot lie on a real line. Indeed, the intersection number of a line with an oval is even; hence such line  would have to intersect $X$ in at least $4+2 = 6$ points (counted with multiplicity) which is impossible. 

       Realizations of the divisors corresponding to the degree partitions $(1,3)$ and $(3,2)$ can be found in~\cite[Example 3.8]{kummer_separating_2020}.  
    \end{enumerate}

  \section{Proof of the main theorem}\label{main_theorem_proof}
  
  Now we are ready to prove the main theorem: 

  \begin{theorem}\label{main_thm}
    For a non-negative integer number $g$ the set of all separating semigroups of genus $g$ curves is finite. 
  \end{theorem}
  \begin{proof} 
    We may assume that $g \geqslant 2$, as the result for $g \leqslant 1$ follows from the known classification (see the introduction). By Harnack's inequality, any real curve $X$ of genus $g$ satisfies $b_0(\R X) \leqslant g + 1$.  Hence, it suffices to prove the claim for fixed $g$ and $r = b_0(\R X)$.

    Let $X$ be a real curve of genus $g$ with $b_0(\R X) = r$.  By Lemma~\ref{main_lemma}, its separating semigroup admits the decomposition
    \[
      \Sep(X) = \Sep_s(X) \,\cup\, \bigcup_{i=1}^{m} \bigl( d(P_i) + \N_0^{\,r} \bigr).
    \]

    Since for every $d \in \Sep_s(X)$ we have $|d| \leqslant 2g - 2$, there are only finitely many possibilities for $\Sep_s(X)$. Thus, in order to prove the theorem, it suffices to obtain a uniform bound on the degrees of the minimal non-special divisors $P_1, \ldots, P_m$.

    \begin{lemma}\label{4g-3_lemma}
    For $g \geqslant 2$, one has $\deg P_i \leqslant 4g - 3$.
    \end{lemma}

    \begin{proof}[Proof of Lemma~\ref{4g-3_lemma}]
    Suppose, for a contradiction, that $\deg P_i \geqslant 4g - 2$. 
    Since $g \geqslant 2$, we have $4g - 2 \geqslant g + 2$. 
    Then, by Theorem~\ref{deleting_points} (see Remark~\ref{degree_g+2}), there exists a separating divisor $Q_i$ such that $Q_i \leqslant P_i$, $Q_i \ne P_i$, and
    \[
      \deg Q_i \geqslant \left\lceil \frac{\deg P_i}{2} \right\rceil \geqslant 2g - 1.
    \]
    Hence, $Q_i$ is non-special (as $\deg Q_i \geqslant 2g - 1 > 2g - 2$) and separating, contradicting the minimality of $P_i$.
    \end{proof}

    It follows from Lemma~\ref{4g-3_lemma} that the number $m$ of minimal non-special divisors satisfies 
    \[m \leqslant (4g - 3)^r.\]
      Consequently, for fixed $g$ and $r$, there are only finitely many possible degree partitions $d(P_i)$, which implies the claim of Theorem~\ref{main_thm}.

    \end{proof}

    \begin{remark}
    As the known examples show (see, e.g., item~(3) in \S~\ref{examples}), the bound $\deg P_i \leqslant 4g - 3$ is far from sharp. This leads us to the following question.
    \end{remark}

    \begin{question}
    What is the optimal bound (in terms of $g$ and $r$) for the degree of a minimal non-special separating divisor?
    \end{question}

    \subsection*{Acknowledgements}
    This paper was inspired by a question posed to me by Sergei O. Gorchinskiy after my talk at the Vth Conference of Mathematical Centers of Russia, concerning the description of the separating semigroup in finite terms. Further development of this work, and, in particular, the proof of the main theorem was stimulated by a question from Olivier Benoist, which I received after the publication of the first version of the preprint. I am grateful to both of them for these questions.

    I am deeply indebted to Grigory B. Mikhalkin for his interest in this work, fruitful discussions, and valuable comments. I also thank Fedya Ushakov for helpful conversations and the anonymous referee for their suggestions that significantly improved the exposition of this paper.

    \bibliographystyle{alphaurl}
   \bibliography{reference}

\end{document}